\documentclass[11pt]{amsart}

\usepackage{amssymb,amsmath,amscd}
\usepackage{amsthm,amsfonts}
\usepackage{dsfont}
\usepackage{tikz}
\usepackage{xcolor}
\usepackage{hyperref}
\usepackage[T1]{fontenc}
\usepackage[utf8]{inputenc}

\newcommand{\ud}{\mathrm{d}}
\newcommand{\e}{\mathrm{e}}
\newcommand{\CR}{\mathds{R}}

\newcommand{\CC}{\mathds{C}}
\newcommand{\CN}{\mathds{N}}
\newcommand{\depp}[2]{\frac{\partial\emph{$#1$}}{\partial{\emph{$#2$}}}}
\theoremstyle{plain}
\newtheorem{theorem}{Theorem}[section]
\newtheorem{definition}{Definition}[section]
\newtheorem{lemma}{Lemma}[section]
\newtheorem{proposition}{Proposition}[section]
\newtheorem{corollary}{Corollary}[section]
\newtheorem{conjecture}{Conjecture}[section]
\theoremstyle{definition}
\newtheorem{example}{Example}[section]
\newtheorem{remark}{Remark}[section]

\begin{document}

\title[Geometric Quantization of semitoric systems]{Geometric Quantization of semitoric systems and almost toric manifolds}
\dedicatory{Dedicated to the memory of Bertram Kostant}
\author{Eva Miranda}\address{Eva Miranda, Departament de Matem\`{a}tiques, Universitat Polit\`{e}cnica de Catalunya, and BGSMath Barcelona Graduate School of
Mathematics. Postal address: Avinguda del Doctor Mara\~{n}\'{o}n, 44-50, 08028 Barcelona, Spain. \it{e-mail: eva.miranda@upc.edu}}
 \author{Francisco Presas}\address{Francisco Presas, ICMAT, UAM-UCM-UC3M. Postal address: Nicol\'{a}s Cabrera, 13-15, Campus Cantoblanco UAM, 28049 Madrid, Spain. \it{e-mail: fpresas@icmat.es }}
\author{Romero Solha}\address{Romero Solha, Departamento de Matem\'{a}tica, Pontif\'{i}cia Universidade Cat\'{o}lica do Rio de Janeiro. Postal address: Rua Marqu\^{e}s de São Vicente, 225 - Edif\'{i}cio Cardeal Leme, sala 862, G\'{a}vea, Rio de Janeiro - RJ, Brazil (postal code 22451-900). \it{e-mail: romerosolha@gmail.com }}
\thanks{ Eva Miranda is supported by the Catalan Institution for Research and Advanced Studies via an ICREA Academia 2016 Prize and partially supported by the grants reference number MTM2015-69135-P (MINECO/FEDER) and reference number 2014SGR634 (AGAUR). Romero Solha is supported by CAPES and partially supported by MTM2015-69135-P (MINECO/FEDER). Francisco Presas is supported by the grants reference number MTM2016-79400-P (MINECO/FEDER). Eva Miranda and Francisco Presas are supported by an EXPLORA CIENCIA project with reference number MTM2015-72876-EXP and by the excellence project SEV-2015-0554.}
\date{\today}


\begin{abstract}Kostant gave a model for the real geometric quantization associated to polarizations via the cohomology associated to the sheaf of flat sections of a pre-quantum line bundle. This model is well-adapted for real polarizations given by integrable systems and toric manifolds. In the latter case, the cohomology can be computed counting integral points inside the associated Delzant polytope. In this article we extend Kostant's geometric quantization to semitoric integrable systems and almost toric manifolds. In these cases the dimension of the acting torus is smaller than half of the dimension of the manifold. In particular, we compute the cohomology groups associated to the geometric quantization if the real polarization is the one associated to an integrable system with focus-focus type singularities in dimension four. As application we determine models for the geometric quantization of K3 surfaces, a spin-spin system, the spherical pendulum, and a spin-oscillator system under this scheme.
\end{abstract}

\maketitle

\section{Introduction}

An important contribution of Kostant has been the definition of geometric quantization via the cohomology associated to the sheaf of sections of a chosen pre-quantum line bundle that are flat along a given polarization. This construction using real polarizations is an abstraction of K\"{a}hler quantization and has been used in connection to representation theory (see for instance \cite{GuSt}). Generalizations of this scheme considering non-degenerate singularities have also been obtained by Hamilton \cite{Ha}, Hamilton and Miranda \cite{HaMi}, and Solha \cite{Solha}.

A toric manifold is a symplectic manifold endowed with an effective Hamiltonian action of a torus whose rank is half of the dimension of the manifold. A theorem of Delzant \cite{Del} establishes a one-to-one correspondence between closed toric manifolds in dimension $2m$ and a class of polytopes (called Delzant polytopes) on $\CR^m$. The real geometric quantization of closed toric manifolds can be read from the Delzant polytope, as it was proved by Hamilton \cite{Ha} (generalizing previous results by \'{S}niatycki \cite{Sni} to the singular context): given a toric manifold, its real geometric quantization is determined by the sheer count of integral points inside (boundary points are excluded) its associated Delzant polytope.

Toric manifolds are central in the study of the geometry of symplectic manifolds and their symmetries, as well as their generalizations, like semitoric integrable systems \cite{pelayo4} or almost toric manifolds \cite{LeSy}, in which the rank of the torus is no longer half of the dimension of the manifold. Examples of almost toric manifolds are given by K3 surfaces, which are also of relevance in complex geometry. A semitoric integrable system (see for instance \cite{pelayo3, pelayo4}) is an integrable system admitting only non-degenerate singularities composed of elliptic and focus-focus types, but excluding any hyperbolic components.

As observed in \cite{Solha} the quantization of almost toric manifolds can be reduced to the computation of the contribution of a neighborhood of a Bohr--Sommerfeld focus-focus singular fiber by the use of factorization tools\footnote{Which behave following a K\"{u}nneth formula \cite{MiPr}, as a simple sheaf cohomology.}. This is because the geometric quantization of neighborhoods of Bohr--Sommerfeld fibers computes the geometric quantization of the whole manifold (by means of a standard Mayer--Vietoris sequence).

In this article we associate Kostant's model to focus-focus singularities and conclude its computation showing that the first cohomology group associated to the real geometric quantization of a small neighborhood of a focus-focus fiber of a 4-dimensional semitoric integrable system is trivial, but not the second cohomology group, which is infinite dimensional when the singular fiber is Bohr--Sommerfeld. This determines completely the geometric quantization when the real polarization has focus-focus fibers (the cohomology group in degree zero is trivial and had already been computed in \cite{Solha}) and thus closes up the problem of geometric quantization of integrable systems with non-degenerate singularities as initiated in \cite{Ha} and \cite{HaMi} for $4$-dimensional manifolds with no hyperbolic-hyperbolic fibers.

As a motivation for these results, we present K3 surfaces as an example of almost toric manifolds and analyze the effect of nodal trades \cite{LeSy} in their real quantization. Other models of quantization for K3 surfaces have been recently obtained by Castejón \cite{castejon} using the Berezin--Toepliz operators approach \cite{bms}. For this direction see also \cite{pelayo1}. We also analyse the examples of a spin-spin system, the spherical pendulum, and a spin-oscillator system.


\section{Main definitions}

\subsection{Singular Lagrangian fibrations}

The symplectic manifolds of interest to this article have a great deal of symmetry, and such symmetries are related to some particular classes of integrable systems: the ones admitting only non-degenerate singularities.

\begin{definition}An integrable system on a symplectic manifold $(M,\omega)$ of dimension $2m$ is a set of $m$ functions, $f_1,\dots,f_m\in C^\infty(M;\CR)$, satisfying
\begin{equation*}
\ud f_1\wedge\cdots\wedge\ud f_m\neq 0 \text{ over an open dense subset of $M$ and}
\end{equation*}
\begin{equation*}
\{f_j,f_k\}_\omega=0 \text{ for all $j,k$.}
\end{equation*}
\end{definition}

The \emph{Poisson bracket} is defined by $\{f,\cdot\}_\omega=X_f(\cdot)$, where $X_f$ is the unique vector field defined by the equation $\imath_{X_f}\omega+\ud f=0$, called the \emph{Hamiltonian vector field} of $f$.

The next definition refers to the critical set of an integrable system, i.e. the set of points where $\ud f_1\wedge\cdots\wedge\ud f_m$ vanishes.

\begin{definition}A critical point of rank $k_r=m-k_e-k_h-2k_f$ of an integrable system $(f_1,\dots,f_m):M\to\CR^m$ is a non-degenerate singular point of Williamson type $(k_e,k_h,k_f)$ if the quadratic parts of $f_1,\dots,f_m$ can be written as:
\begin{equation*}\begin{array}{lc}
\quad h_j=x_j \phantom{lmmmmmmmmm} \text{(regular)} & 1 \leq j \leq k_r \\
\quad h_j=x_j^2+y_j^2 \phantom{nmmmmmm} \text{(elliptic)}& k_r+1 \leq j \leq k_r+k_e \\
\quad h_j=x_jy_j \phantom{lnmmmmmmm} \text{(hyperbolic)} & k_r+k_e+1 \leq j \leq k_r+k_e+k_h \\
\left\{\begin{array}{l} h_j=x_j y_j+x_{j+1} y_{j+1} \\ h_{j+1}=x_j y_{j+1}-x_{j+1} y_j \end{array}\right. \text{(focus-focus)} & \begin{array}{c}
j=k_r+k_e+k_h+2l-1,\\ 1\leq l \leq k_f \end{array}
\end{array}
\end{equation*}in some Darboux local coordinates $(x_1,y_1,\dots,x_m,y_m)$.
\end{definition}

\begin{example}The stable equilibrium point in the simple pendulum is an elliptic singularity ($k_e=1$, $k_h=0$, and $k_f=0$), while the unstable one is of hyperbolic type ($k_e=0$, $k_h=1$, and $k_f=0$). The spherical pendulum has a stable equilibrium point which is a purely elliptic singularity ($k_e=2$, $k_h=0$, and $k_f=0$), the unstable equilibrium point is a focus-focus singularity ($k_e=0$, $k_h=0$, and $k_f=1$).
\end{example}

A singular fiber of an integrable system is said to be of \emph{Williamson type} $(k_e,k_h,k_f)$ if all of its singular points are non-degenerate singular points of that same Williamson type. Another terminology is also used in this article: in dimension 2 an \emph{elliptic fiber} and a \emph{hyperbolic fiber} are singular fibers of Williamson type $(1,0,0)$ and $(0,1,0)$, in dimension 4 a \emph{focus-focus fiber} is a singular fiber of Williamson type $(0,0,1)$.

When we refer to a \emph{foliation associated to an integrable system} we refer to the foliation described by the orbits of the Hamiltonian vector fields. When we refer to a \emph{fibration associated to an integrable system} we refer to the foliation defined by the fibers. Those two foliations do not necessarily coincide at the singular points. As it was proved by Eliasson \cite{eliasson,eliasson1} and Miranda \cite{Mi, miranda, MiZu} non-degenerate singularities are characterized by the fact that the foliation associated to an integrable system is equivalent to the foliation described by its quadratic part.

\begin{theorem} The foliation of the integrable system given by $f_1,\dots,f_m$ is locally equivalent, in a neighborhood of a non-degenerate singular point of Williamson type $(k_e,k_h,k_f)$, to the foliation of the integrable system given by:
\begin{equation*}\begin{array}{lc}
\quad h_j=x_j \phantom{lmmmmmmmmm} \text{(regular)} & 1 \leq j \leq k_r \\
\quad h_j=x_j^2+y_j^2 \phantom{nmmmmmm} \text{(elliptic)}& k_r+1 \leq j \leq k_r+k_e \\
\quad h_j=x_jy_j \phantom{lnmmmmmmm} \text{(hyperbolic)} & k_r+k_e+1 \leq j \leq k_r+k_e+k_h \\
\left\{\begin{array}{l} h_j=x_j y_j+x_{j+1} y_{j+1} \\ h_{j+1}=x_j y_{j+1}-x_{j+1} y_j \end{array}\right. \text{(focus-focus)} & \begin{array}{c}
j=k_r+k_e+k_h+2l-1,\\ 1\leq l \leq k_f \end{array}
\end{array}
\end{equation*}in some Darboux local system of coordinates $(x_1,y_1,\dots,x_m,y_m)$.
\end{theorem}

Let us define the notion of singular Lagrangian fibration.

\begin{definition}A singular Lagrangian fibration is a symplectic manifold $(M,\omega)$ of dimension $2m$ together with a surjective map $\mathcal{F}:M\to N$, where $N$ is a topological space of dimension $m$, such that for every point in $N$ there exist an open neighborhood $V\subset N$ and a homeomorphism $\chi:V\to U\subset\CR^m$ satisfying that $\chi\circ\mathcal{F}\big{|}_{\mathcal{F}^{-1}(V)}$ is an integrable system on $(\mathcal{F}^{-1}(V),\omega\big{|}_{\mathcal{F}^{-1}(V)})$.
\end{definition}

When the integrable systems in a singular Lagrangian fibration do not have singularities, one refers to it as a \emph{regular Lagrangian fibration}. The real geometric quantization of such manifolds were computed in \cite{Sni}, whereas (closed) \emph{locally toric manifolds} were considered in \cite{Ha} (see \cite{Solha} for the non-compact case); these symplectic manifolds are singular Lagrangian fibrations whose singularities are only of Williamson type $(k_e,0,0)$.

\emph{Almost toric manifolds} are singular Lagrangian fibrations admitting only singularities of Williamson type $(k_e,0,k_f)$. In particular, regular Lagrangian fibrations and locally toric manifolds (which include toric manifolds), are examples of almost toric manifolds; as well as the \emph{semitoric integrable systems} in dimension four, which are included in the almost toric manifolds whose bases are subsets of $\CR^2$. The semi-local and global classification of these symplectic manifolds have been the object of study of \cite{chaperon, LeSy, pelayo3, pelayo4, san}.

\subsection{Real geometric quantization}

Let $(M,\omega)$ be a symplectic manifold of dimension $2m$ whose de Rham class $[\omega]$ admits an integral lift. Such a symplectic manifold will be called \emph{pre-quantizable}.

\begin{definition}\label{prequantumdef}A pre-quantum line bundle for $(M,\omega)$ is a complex line bundle $L$ over $M$ with a connection $\nabla^{\omega}$ satisfying $curv(\nabla^{\omega})=-i\omega$.
\end{definition}

\begin{definition}\label{polarisationdef} A real polarization $\mathcal{P}$ is an integrable (in the Sussmann's sense) distribution of $TM$ whose leaves are generically Lagrangian. The complexification of $\mathcal{P}$ is denoted by $P$ and will be called polarization.
\end{definition}

The most relevant real polarization for this work is $\langle X_{f_1},...,X_{f_m}\rangle_{C^\infty(M;\CR)}$: the distribution of the Hamiltonian vector fields of an integrable system. The leaves of the associated (possibly singular) foliation are isotropic submanifolds and they are Lagrangian at points where the first integrals are functionally independent.

\begin{definition} Let $\mathcal{J}$ denote the sheaf of sections of a pre-quantum line bundle $L$ such that for each open set $V\subset M$ the set $\mathcal{J}(V)$ is the module (over the ring of smooth leafwise constant complex-valued functions of $V$) of sections $s\in L$ defined over $V$ satisfying $\nabla^{\omega}_X s=0$ for all vector fields $X$ in $P$ defined over $V$.
\end{definition}

\begin{definition} The quantization of $(M,\omega,L,\nabla^{\omega},P)$ is given by
\begin{equation*}
\mathcal{Q}(M)=\displaystyle\bigoplus_{n\geq 0}H^n(M;\mathcal{J}) \ ,
\end{equation*}where $H^n(M;\mathcal{J})$ are the sheaf cohomology groups associated to $\mathcal{J}$.
\end{definition}

The following definition plays a very important role in the computation of the cohomology groups appearing in geometric quantization:

\begin{definition}A leaf $\ell$ of $\mathcal{P}$ is Bohr--Sommerfeld if there exists a non-vanishing section $s:\ell\to L$ such that $\nabla^\omega_X s=0$ for any complex vector field $X$ in the polarization $P$ (restricted to $\ell$). Fibers that are a union of Bohr--Sommerfeld leaves are called Bohr--Sommerfeld fibers.
\end{definition}


\section{A motivating example: K3 surfaces}\label{KKK}

A K3 surface is an example of a total space of an almost toric manifold: it admits an almost toric fibration over the sphere with 24 focus-focus fibers. The base is a sphere with 24 marked points, and on the complement of these points one has a regular Lagrangian fibration with torus fibers.

A way to construct such an almost toric manifold, as done in \cite{LeSy} (cf. \cite{Gom}), is to consider two copies of a (symplectic and toric) blowup of the complex projective plane at 9 different points as toric manifolds, apply nodal trades to all of their elliptic-elliptic singular fibers, and take their symplectic sum along the symplectic tori corresponding to the preimage of the boundary of their respective bases (as almost toric fibrations). Starting with a pre-quantizable K3 surface, this construction (together with a gluing result described in subsection \ref{glue}) allows one to obtain a K3 surface with up to 24 Bohr--Sommerfeld focus-focus fibers.

Here is how the construction works.

\begin{itemize}
\item Starting with a complex projective plane, understood as a toric manifold and described here by its Delzant polytope \cite{Del}, one performs three blowups at different points, represented in their Delzant polytopes by cuts based on their three vertices (cf. \cite{KKP}), followed by another six blowups at different points, represented in their Delzant polytopes by cuts based on their six new vertices formed after the first three blowups: see figure \ref{figK3}.

\begin{figure}[h]
\centering
\setlength{\unitlength}{1em}
\begin{picture}(11,9)
\put(8,4){\vector(1,0){2}}
\put(0,0){\circle*{0,2}}
\put(0,9){\circle*{0,2}}
\put(9,0){\circle*{0,2}}
\multiput(0,0)(1,0){10}{\circle*{0.1}}
\multiput(0,1)(1,0){9}{\circle*{0.1}}
\multiput(0,2)(1,0){8}{\circle*{0.1}}
\multiput(0,3)(1,0){7}{\circle*{0.1}}
\multiput(0,4)(1,0){6}{\circle*{0.1}}
\multiput(0,5)(1,0){5}{\circle*{0.1}}
\multiput(0,6)(1,0){4}{\circle*{0.1}}
\multiput(0,7)(1,0){3}{\circle*{0.1}}
\multiput(0,8)(1,0){2}{\circle*{0.1}}
\multiput(0,9)(1,0){1}{\circle*{0.1}}
\put(0,0){\line(1,0){9}}
\multiput(3,0)(-0.2,0.2){15}{\circle*{0.1}} 
\put(9,0){\line(-1,1){9}}
\multiput(6,0)(0,0.2){15}{\circle*{0.1}} 
\put(0,9){\line(0,-1){9}}
\multiput(0,6)(0.2,0){15}{\circle*{0.1}} 
\end{picture}
\begin{picture}(10,9)
\put(7,4){\vector(1,0){2}}
\put(3,0){\circle*{0,2}}
\put(6,0){\circle*{0,2}}
\put(0,3){\circle*{0,2}}
\put(6,3){\circle*{0,2}}
\put(0,6){\circle*{0,2}}
\put(3,6){\circle*{0,2}}
\multiput(3,0)(1,0){4}{\circle*{0.1}}
\multiput(2,1)(1,0){5}{\circle*{0.1}}
\multiput(1,2)(1,0){6}{\circle*{0.1}}
\multiput(0,3)(1,0){7}{\circle*{0.1}}
\multiput(0,4)(1,0){6}{\circle*{0.1}}
\multiput(0,5)(1,0){5}{\circle*{0.1}}
\multiput(0,6)(1,0){4}{\circle*{0.1}}
\put(3,0){\line(1,0){3}}
\put(3,0){\line(-1,1){3}} 
\multiput(4,0)(-0.2,0.1){11}{\circle*{0.1}} 
\put(6,3){\line(-1,1){3}}
\put(6,0){\line(0,1){3}} 
\multiput(5,0)(0.2,0.2){6}{\circle*{0.1}} 
\multiput(6,2)(-0.1,0.2){11}{\circle*{0.1}}
\put(0,6){\line(0,-1){3}}
\put(0,6){\line(1,0){3}} 
\multiput(4,5)(-0.2,0.1){11}{\circle*{0.1}} 
\multiput(0,5)(0.2,0.2){6}{\circle*{0.1}} 
\multiput(0,4)(0.1,-0.2){11}{\circle*{0.1}} 
\end{picture}
\begin{picture}(6,9)
\put(4,0){\circle*{0,2}}
\put(5,0){\circle*{0,2}}
\put(1,2){\circle*{0,2}}
\put(6,2){\circle*{0,2}}
\put(0,5){\circle*{0,2}}
\put(2,6){\circle*{0,2}}
\put(2,1){\circle*{0,2}}
\put(6,1){\circle*{0,2}}
\put(0,4){\circle*{0,2}}
\put(5,4){\circle*{0,2}}
\put(1,6){\circle*{0,2}}
\put(4,5){\circle*{0,2}}
\multiput(4,0)(1,0){2}{\circle*{0.1}}
\multiput(2,1)(1,0){5}{\circle*{0.1}}
\multiput(1,2)(1,0){6}{\circle*{0.1}}
\multiput(1,3)(1,0){5}{\circle*{0.1}}
\multiput(0,4)(1,0){6}{\circle*{0.1}}
\multiput(0,5)(1,0){5}{\circle*{0.1}}
\multiput(1,6)(1,0){2}{\circle*{0.1}}
\put(4,0){\line(1,0){1}}
\put(2,1){\line(-1,1){1}} 
\put(4,0){\line(-2,1){2}} 
\put(5,4){\line(-1,1){1}}
\put(6,1){\line(0,1){1}} 
\put(5,0){\line(1,1){1}} 
\put(6,2){\line(-1,2){1}} 
\put(0,5){\line(0,-1){1}}
\put(1,6){\line(1,0){1}} 
\put(4,5){\line(-2,1){2}} 
\put(0,5){\line(1,1){1}} 
\put(0,4){\line(1,-2){1}} 
\end{picture}
\caption{$\CC P^2$, $\CC P^2\# 3\overline{\CC P}^2$, and $\CC P^2\# 9\overline{\CC P}^2$.}\label{figK3}
\end{figure}
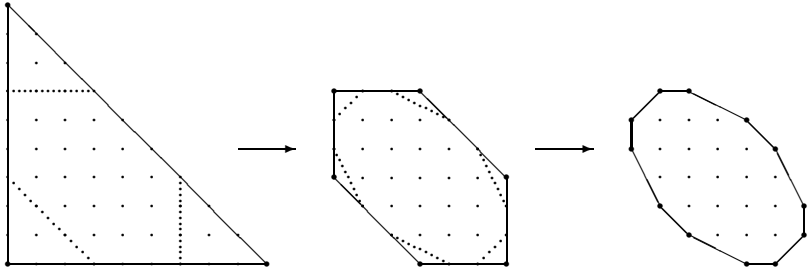

\item Now, following \cite{LeSy}, one can perform nodal trades to all the vertices of the resulting Delzant polytope. In figure \ref{figK32} each nodal trade is being represented by a vector based at a vertex, and the monodromy around each of the resulting focus-focus fibers can be read from those vectors.

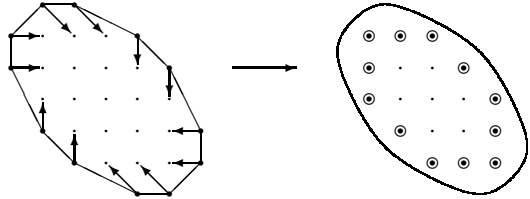
\begin{figure}[h]
\centering
\setlength{\unitlength}{1.1em}
\begin{picture}(10,7)
\put(7,4){\vector(1,0){2}}
\put(4,0){\circle*{0,2}}
\put(5,0){\circle*{0,2}}
\put(1,2){\circle*{0,2}}
\put(6,2){\circle*{0,2}}
\put(0,5){\circle*{0,2}}
\put(2,6){\circle*{0,2}}
\put(2,1){\circle*{0,2}}
\put(6,1){\circle*{0,2}}
\put(0,4){\circle*{0,2}}
\put(5,4){\circle*{0,2}}
\put(1,6){\circle*{0,2}}
\put(4,5){\circle*{0,2}}
\multiput(4,0)(1,0){2}{\circle*{0.1}}
\multiput(2,1)(1,0){5}{\circle*{0.1}}
\multiput(1,2)(1,0){6}{\circle*{0.1}}
\multiput(1,3)(1,0){5}{\circle*{0.1}}
\multiput(0,4)(1,0){6}{\circle*{0.1}}
\multiput(0,5)(1,0){5}{\circle*{0.1}}
\multiput(1,6)(1,0){2}{\circle*{0.1}}
\put(0,4){\vector(1,0){0.9}} 
\put(6,1){\vector(-1,0){0.9}} 
\put(0,5){\vector(1,0){0.9}} 
\put(6,2){\vector(-1,0){0.9}} 
\put(1,2){\vector(0,1){0.9}} 
\put(4,5){\vector(0,-1){0.9}} 
\put(2,1){\vector(0,1){0.9}} 
\put(5,4){\vector(0,-1){0.9}} 
\put(1,6){\vector(1,-1){0.9}} 
\put(4,0){\vector(-1,1){0.9}} 
\put(2,6){\vector(1,-1){0.9}} 
\put(5,0){\vector(-1,1){0.9}} 
\put(4,0){\line(1,0){1}}
\put(2,1){\line(-1,1){1}} 
\put(4,0){\line(-2,1){2}} 
\put(5,4){\line(-1,1){1}}
\put(6,1){\line(0,1){1}} 
\put(5,0){\line(1,1){1}} 
\put(6,2){\line(-1,2){1}} 
\put(0,5){\line(0,-1){1}}
\put(1,6){\line(1,0){1}} 
\put(4,5){\line(-2,1){2}} 
\put(0,5){\line(1,1){1}} 
\put(0,4){\line(1,-2){1}} 
\end{picture}
\begin{picture}(7,9)
\multiput(3,1)(1,0){3}{\circle*{0.1}}
\multiput(2,2)(1,0){4}{\circle*{0.1}}
\multiput(1,3)(1,0){5}{\circle*{0.1}}
\multiput(1,4)(1,0){4}{\circle*{0.1}}
\multiput(1,5)(1,0){3}{\circle*{0.1}}
\put(1,4){\circle*{0.2}}
\put(5,1){\circle*{0.2}}
\put(1,5){\circle*{0.2}}
\put(5,2){\circle*{0.2}}
\put(1,3){\circle*{0.2}}
\put(4,4){\circle*{0.2}}
\put(2,2){\circle*{0.2}}
\put(5,3){\circle*{0.2}}
\put(2,5){\circle*{0.2}}
\put(3,1){\circle*{0.2}}
\put(3,5){\circle*{0.2}}
\put(4,1){\circle*{0.2}}
\put(1,4){\circle{0.3}}
\put(5,1){\circle{0.3}}
\put(1,5){\circle{0.3}}
\put(5,2){\circle{0.3}}
\put(1,3){\circle{0.3}}
\put(4,4){\circle{0.3}}
\put(2,2){\circle{0.3}}
\put(5,3){\circle{0.3}}
\put(2,5){\circle{0.3}}
\put(3,1){\circle{0.3}}
\put(3,5){\circle{0.3}}
\put(4,1){\circle{0.3}}
\qbezier(4.5,0)(5,0)(5.5,0.5) 
\qbezier(5.5,0.5)(6,1)(6,1.5) 
\qbezier(6,1.5)(6,2)(5.5,3) 
\qbezier(5.5,3)(5,4)(4.5,4.5) 
\qbezier(4.5,4.5)(4,5)(3,5.5) 
\qbezier(3,5.5)(2,6)(1.5,6) 
\qbezier(1.5,6)(1,6)(0.5,5.5) 
\qbezier(0.5,5.5)(0,5)(0,4.5) 
\qbezier(0,4.5)(0,4)(0.5,3) 
\qbezier(0.5,3)(1,2)(1.5,1.5) 
\qbezier(1.5,1.5)(2,1)(3,0.5) 
\qbezier(3,0.5)(4,0)(4.5,0) 
\end{picture}
\caption{Nodal trades on $\CC P^2\# 9\overline{\CC P}^2$.}\label{figK32}
\end{figure}

\item The resulting manifold, $\CC P^2\# 9\overline{\CC P}^2$, is endowed with an almost toric fibration, and the preimage of the boundary of its base is a symplectic torus. Thus, one can consider two copies of this symplectic manifold and perform a symplectic sum along these tori (cf. \cite{Gom}), obtaining a K3 surface, $(\CC P^2\# 9\overline{\CC P}^2)\#_{T^2}(\CC P^2\# 9\overline{\CC P}^2)$; together with an almost toric fibration whose base is the sphere formed by gluing two copies of the previously constructed disk (with its twelve marked points) along their boundary (see figure \ref{figK33}).

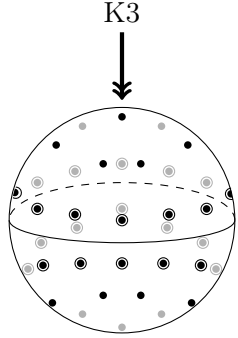
\begin{figure}[h]
\centering
\setlength{\unitlength}{1em}
\begin{tikzpicture}
\node [above] at (0,2.5) {K3};
\begin{scope}[scale=0.5]
\draw[ultra thick,->] (0,5)--(0,3.2);
\draw[ultra thick,->] (0,3.6)--(0,3.4);
\fill[fill=black!30] (0,0.3) circle (0.1);
\draw[black!30] (0,0.3) circle (0.15);
\fill[fill=black!30] (-2.15,-0.6) circle (0.1);
\draw[black!30] (-2.15,-0.6) circle (0.15);
\fill[fill=black!30] (-1.2,-0.2) circle (0.1);
\draw[black!30] (-1.2,-0.2) circle (0.15);
\fill[fill=black!30] (1.2,-0.2) circle (0.1);
\draw[black!30] (1.2,-0.2) circle (0.15);
\fill[fill=black!30] (2.15,-0.6) circle (0.1);
\draw[black!30] (2.15,-0.6) circle (0.15);
\fill[fill=black!30] (-2.5,-1.3) circle (0.1);
\draw[black!30] (-2.5,-1.3) circle (0.15);
\fill[fill=black!30] (2.5,-1.3) circle (0.1);
\draw[black!30] (2.5,-1.3) circle (0.15);
\fill[fill=black!30] (0,1.5) circle (0.1);
\draw[black!30] (0,1.5) circle (0.15);
\fill[fill=black!30] (-2.25,1) circle (0.1);
\draw[black!30] (-2.25,1) circle (0.15);
\fill[fill=black!30] (-1.25,1.3) circle (0.1);
\draw[black!30] (-1.25,1.3) circle (0.15);
\fill[fill=black!30] (1.25,1.3) circle (0.1);
\draw[black!30] (1.25,1.3) circle (0.15);
\fill[fill=black!30] (2.25,1) circle (0.1);
\draw[black!30] (2.25,1) circle (0.15);
\fill[fill=black!30] (-1.05,-2.5) circle (0.1);
\fill[fill=black!30] (1.05,-2.5) circle (0.1);
\fill[fill=black!30] (0,-2.85) circle (0.1);
\fill[fill=black!30] (-1.05,2.5) circle (0.1);
\fill[fill=black!30] (1.05,2.5) circle (0.1);
\fill[fill=black] (0.5,1.5) circle (0.1);
\fill[fill=black] (-0.5,1.5) circle (0.1);
\fill[fill=black] (0,2.75) circle (0.1);
\fill[fill=black] (1.75,2) circle (0.1);
\fill[fill=black] (-1.75,2) circle (0.1);
\fill[fill=black] (0.5,-2) circle (0.1);
\fill[fill=black] (-0.5,-2) circle (0.1);
\fill[fill=black] (1.85,-2.2) circle (0.1);
\fill[fill=black] (-1.85,-2.2) circle (0.1);
\fill[fill=black] (0,-1.15) circle (0.1);
\draw (0,-1.15) circle (0.15);
\fill[fill=black] (-2.1,-1.2) circle (0.1);
\draw (-2.1,-1.2) circle (0.15);
\fill[fill=black] (-1.1,-1.15) circle (0.1);
\draw (-1.1,-1.15) circle (0.15);
\fill[fill=black] (1.1,-1.15) circle (0.1);
\draw (1.1,-1.15) circle (0.15);
\fill[fill=black] (2.1,-1.2) circle (0.1);
\draw (2.1,-1.2) circle (0.15);
\fill[fill=black] (0,0) circle (0.1);
\draw (0,0) circle (0.15);
\fill[fill=black] (-2.25,0.3) circle (0.1);
\draw (-2.25,0.3) circle (0.15);
\fill[fill=black] (-1.25,0.15) circle (0.1);
\draw (-1.25,0.15) circle (0.15);
\fill[fill=black] (1.25,0.15) circle (0.1);
\draw (1.25,0.15) circle (0.15);
\fill[fill=black] (2.25,0.3) circle (0.1);
\draw (2.25,0.3) circle (0.15);
\draw (-2.9,0.6) arc (-120:90:0.15);
\fill[fill=black] (-2.9,0.65) arc (-120:90:0.1);
\draw (2.9,0.85) arc (90:290:0.15);
\fill[fill=black] (2.9,0.8) arc (90:290:0.1);
\draw (0,0) circle (3);
\draw (-3,0) arc (180:360:3 and 0.6);
\draw[dashed] (3,0) arc (0:180:3 and 1);
\end{scope}
\end{tikzpicture}
\caption{K3 surface as a singular fiber bundle over the sphere.}\label{figK33}
\end{figure}

\end{itemize}

Before the nodal trades, the toric manifold $\CC P^2\# 9\overline{\CC P}^2$ of figure \ref{figK3} admits a pre-quantum line bundle such that all the integer lattice points belonging to its Delzant polytope are images of Bohr--Sommerfeld fibers \cite{GuSt,Solha}. Nodal trades produce a one parameter family of symplectomorphic manifolds (via nodal slides \cite{LeSy}), and the almost toric manifold constructed in figure \ref{figK32} is related by a symplectomorphism isotopic to the identity (the same is true for different nodal trades resulting in up to 12 focus-focus fibers outside the integer lattice).
Therefore, it inherits a pre-quantum line bundle whose Bohr--Sommerfeld fibers are still given by the integer lattice points in the base (which now includes up to 12 focus-focus fibers, depending on the size of the nodal trades).

When gluing two copies of those almost toric manifolds by a symplectic sum, one can glue the pre-quantum line bundles to obtain a pre-quantum line bundle on a K3 surface having any number (between 0 and 24) of focus-focus Bohr--Sommerfeld fibers. This last assertion is justifyied by applying lemma \ref{Franlemma} and corollary \ref{Francoro} (which can be found, together with their proofs, in subsection \ref{glue}).


\section{Preliminary results}

The first three subsections of this section collect results from the literature needed for the proof of the main theorems of this article, and they are included here for the convenience of the reader.


\subsection{Poincar\'{e} lemmata}

Given a pre-quantizable symplectic manifold $(M,\omega)$ with polarization $P$ and pre-quantum line bundle $(L,\nabla^{\omega})$, it is possible to construct a fine resolution for the sheaf of flat sections $\mathcal{J}$, even when $P$ has non-degenerate singularities \cite{JHR,Msolha2,Solha}.

Let $\Omega_P^n(M)$ denote the set of multi-linear maps
\begin{equation*}
\mathrm{Hom}_{C^\infty(M;\CC)}(\wedge^n_{C^\infty(M;\CC)}P; C^\infty(M;\CC)) \ ,
\end{equation*}usually called polarized $n$-forms, and
\begin{equation*}
S_P^n(L)=\Omega_P^n(M)\otimes_{C^\infty(M;\CC)}\Gamma(L) \ .
\end{equation*}Then, the set of \emph{line bundle valued polarized forms} is
\begin{equation*}
{S_P}^\bullet(L)=\displaystyle\bigoplus_{n\geq 0}S_P^n(L) \ .
\end{equation*}

Therefore, $\nabla=\nabla^{\omega}\big{|}_P:S^0_P(L)\to S^1_P(L)$, the restriction of the connection $\nabla^{\omega}$ to the polarization, extends to a derivation of degree $+1$ on the space of line bundle valued polarized forms: if $\alpha\in\Omega_P^n(M)$ and $s\in\Gamma(L)$,
\begin{equation*}
\ud^\nabla(\alpha\otimes s)=\ud_P\alpha\otimes s+(-1)^n\alpha\wedge\nabla s \ ,
\end{equation*}with the exterior derivative $\ud_P$ being the restriction of the de Rham differential to the directions of the polarization.

Since $\omega=i \ curv(\nabla^\omega)$ vanishes along $P$, $\ud^\nabla$ is a coboundary operator.

If $\mathcal{S}_P^n$ denotes the associated sheaf of $S_P^n(L)$, one can extend $\ud^\nabla$ to a homomorphism of sheaves, $\ud^\nabla:\mathcal{S}_P^n\to\mathcal{S}_P^{n+1}$. The sheaf $\mathcal{S}$ of sections of the line bundle $L$ is isomorphic to $\mathcal{S}_P^0$, and $\mathcal{J}$ is isomorphic to the kernel of $\nabla:\mathcal{S}\to\mathcal{S}_P^1$, understood as a map between sheaves. The associated complex
\begin{equation*}
0\longrightarrow\mathcal{J}\hookrightarrow\mathcal{S}\stackrel{\nabla}{\longrightarrow}\mathcal{S}_P^1\stackrel{\ud^\nabla}{\longrightarrow}\cdots\stackrel{\ud^\nabla}{\longrightarrow}\mathcal{S}_P^m\stackrel{\ud^\nabla}{\longrightarrow}0
\end{equation*}is called the \emph{Kostant complex}, and its cohomology is denoted by $H^\bullet({S_P}^\bullet(L))$.

\begin{theorem}\label{fineresolution1} The Kostant complex is a fine resolution for $\mathcal{J}$ when $\mathcal{P}$ is a subbundle of $TM$, or when $\mathcal{P}$ is induced by an integrable system whose moment map has only non-degenerate singularities. Therefore, each of its cohomology groups, $H^n({S_P}^\bullet(L))$, is isomorphic to $H^n(M;\mathcal{J})$.
\end{theorem}

This theorem is proved in \cite{JHR,Msolha2,Solha} by showing that Poincar\'{e} lemmata exist for the Kostant complex. One Poincar\'{e} lemma of particular importance to the present work is the following:

\begin{theorem}[Solha \cite{Solha}]\label{poincaresolha}The cohomology groups $H^n({S_P}^\bullet(L))$ vanish for all $n$ in any sufficiently small contractible open neighborhood of a focus-focus singularity.
\end{theorem}

\begin{remark}\label{metarema} The only property of $L$ being used here is the existence of flat connections along $P$; thus, the results here work if metaplectic correction is considered.
\end{remark}


\subsection{Mayer--Vietoris sequence}

Analogously to the de Rham cohomology case, there exists a Mayer--Vietoris sequence for $\mathcal{J}$. One can construct, for each pair of open subsets $V$ and $W$ of $M$, the injective homomorphism
\begin{equation*}
R_{V,W}:\mathcal{S}_P^n(V\cup W)\to\mathcal{S}_P^n(V)\oplus\mathcal{S}_P^n(W)
\end{equation*}defined by $R_{V,W}(\boldsymbol{\zeta})=\boldsymbol{\zeta}\big{|}_{V}\oplus\boldsymbol{\zeta}\big{|}_{W}$ and the surjective homomorphism
\begin{equation*}
R_{V,V\cap W}-R_{W,V\cap W}:\mathcal{S}_P^n(V)\oplus\mathcal{S}_P^n(W)\to\mathcal{S}_P^n(V\cap W)
\end{equation*}defined by $R_{V,V\cap W}-R_{W,V\cap W}(\boldsymbol{\alpha}\oplus\boldsymbol{\beta})=\boldsymbol{\alpha}\big{|}_{V\cap W}-\boldsymbol{\beta}\big{|}_{V\cap W}$. The injectivity of $R_{V,W}$ is due to the local identity property of the sheaves, while the surjectivity of $R_{V,V\cap W}-R_{W,V\cap W}$ comes from the existence of partitions of unity for $\mathcal{S}^n_P$.

Thanks to the gluing condition of the sheaves, the image of $R_{V,W}$ is equal to the kernel of $R_{V,V\cap W}-R_{W,V\cap W}$, and the long exact sequence associated to the short exact sequences
\begin{equation*}
0\rightarrow\mathcal{S}_P^n(V\cup W)\hookrightarrow\mathcal{S}_P^n(V)\oplus\mathcal{S}_P^n(W)\twoheadrightarrow\mathcal{S}_P^n(V\cap W)\rightarrow 0 \
\end{equation*}yields the Mayer--Vietoris sequence for the Kostant complex (or $\mathcal{J}$, since the Kostant complex is a resolution for $\mathcal{J}$).


\subsection{K\"{u}nneth formula}

The classical K\"{u}nneth formula also holds for the geometric quantization scheme \cite{MiPr}. Let $(M_1, \mathcal{P}_1)$ and $(M_2, \mathcal{P}_2)$ be a pair of pre-quantizable symplectic manifolds endowed with Lagrangian foliations. The natural Cartesian product for the foliations is Lagrangian with respect to the product symplectic structure. The induced sheaf of flat sections associated to the product foliation will be denoted $\mathcal{J}_{12}$. Note that we use the pre-quantum line bundle defined as pull-backs of the ones defined over $M_1$ and $M_2$.

\begin{theorem}[Miranda and Presas \cite{MiPr}] \label{thm:Kun}
There is an isomorphism
\begin{equation*}
H^n(M_1 \times M_2,\mathcal{J}_{12})\cong\bigoplus_{p+q=n}H^p(M_1, \mathcal{J}_1) \otimes H^q(M_2, \mathcal{J}_2) \ ,
\end{equation*}whenever $M_1$ admits a good cover, the geometric quantization associated to $(M_2,\mathcal{J}_2)$ has finite dimension and $M_2$ is a submanifold of a compact manifold.
\end{theorem}

As an illustration, and to anticipate some needed results, let us mention what is the geometric quantization for $M=T^*I\times (I_s\times S^1)$ with $\omega=\ud x_1\wedge\ud y_1+\ud x_2\wedge\ud y_2$, endowed with a trivial pre-quantum line bundle with connection $\nabla=\ud-i(x_1\ud y_1+x_2\ud y_2)$, and $\mathcal{P}$ generated by $\depp{}{y_1}$ and $\depp{}{y_2}$, where $x_1$ is the coordinate function along the fibers of $T^*I$, $y_1$ is the coordinate function along the open interval $I\subset\CR$, $x_2$ is the coordinate function along the open interval $I_s\subset\CR$, and $y_2$ is the periodic coordinate function along $S^1$.

\begin{proposition}\label{kunn}The geometric quantization of $M=T^*I\times (I_s\times S^1)$, with the extra structures described above, is given by
\begin{equation*}
H^1(T^*I\times (I_s\times S^1);\mathcal{J}_{12})\cong\bigoplus_ {j\in\{1,\dots,n\}}H^0(T^*I;\mathcal{J}_1)\cong\bigoplus_ {j\in\{1,\dots,n\}} C^\infty(\CR;\CC) \ ,
\end{equation*}where $n$ is the number of integers inside $I_s$.
\end{proposition}

The techniques from \cite{MiPr,Sni,Solha} provide isomorphisms from $H^1(T^*I\times (I_s\times S^1);\mathcal{J}_{12})$ to flat sections of trivial pre-quantum line bundles $L$ over $T^*I$, for a given open interval $I\subset\CR$. These flat sections are all of the form
\begin{equation*}
T^*I\ni (x,y)\mapsto h(x)\e^{ixy}s(x,y)\in L|_{(x,y)}\cong\CC
\end{equation*}where $h\in C^\infty(\CR;\CC)$ and $s\in\Gamma(L)$ is a unitary section of $L$ with potential $1$-form $-x\ud y$. For example, theorem \ref{thm:Kun} gives
\begin{equation*}
H^1(T^*I\times (I_s\times S^1);\mathcal{J}_{12})\cong H^0(T^*I;\mathcal{J}_1)\otimes H^1(I_s\times S^1;\mathcal{J}_{2})\cong C^\infty(\CR;\CC)\otimes\CC^n \ .
\end{equation*}


\section{Contribution from focus-focus singularities}

Let $V\subset M$ be an open neighborhood of a non-degenerate focus-focus fiber $\ell_f$ (compact or not) over which a Hamiltonian $S^1$-action is defined \cite{Zu2}. Note that a focus-focus fiber might have more than one singular point (also called a \emph{node}, or \emph{nodal point} \cite{LeSy}).

\begin{lemma}[Solha \cite{Solha}]\label{corosolha}
In the neighborhood of $\ell_f$ over which a Hamiltonian $S^1$-action is defined, there exists a neighborhood $V$ containing only $\ell_f$ as a Bohr--Sommerfeld fiber such that $H^0(V;\mathcal{J}\big{|}_V)=\{0\}$.
\end{lemma}

Therefore, without loss of generality, one can assume that $V$ contains no Bohr--Sommerfeld fiber, or only one if $\ell_f$ is itself Bohr--Sommerfeld. Such a neighborhood $V$ of a focus-focus fiber $\ell_f$ is called a \emph{saturated neighborhood} (since it is saturated by the orbits of the $S^1$-action).

\begin{theorem}\label{ffalive1}The geometric quantization of a saturated neighborhood of a focus-focus fiber with only one node is isomorphic to $C^\infty(\CR;\CC)$ if the focus-focus fiber is compact and Bohr--Sommerfeld, and zero otherwise.
\end{theorem}
\begin{proof}
Let $p\in\ell_f$ be the singular point of the focus-focus fiber, $W\subset V$ be a contractible open neighborhood of the singular point, and $V_0\subset V$ be an open (not connected when $\ell_f$ is not compact) neighborhood with $p\notin V_0$, together satisfying $V=V_0\cup W$ and $V_0\cap W=W^-\sqcup W^+$.

The neighborhood $V$ is the total space of a singular Lagrangian fibration over an open disk $D^2\cong\CR\times I_s$ (with $I_s\subset\CR$ an open interval representing the circle action direction), as well as $W$ (which is diffeomorphic to an open $4$-ball centered in the nodal point $p\in\ell_f$), while $W^-$, $W^+$, and $V_0$ are regular trivial Lagrangian fibrations. Indeed,
\begin{equation*}
V_0\cong (I_0\times S^1)\sqcup (I_{2\pi}\times S^1)\times D^2 \ ,
\end{equation*}with $I_0=(0,b^-)$, $I_{2\pi}=(a^+,2\pi)$, and $a^+>b^-$,
\begin{equation*}
W^-\cong(I^-\times S^1)\times D^2
\end{equation*}with $I^-=(a^-,b^-)$ and $a^->0$, and
\begin{equation*}
W^+\cong(I^+\times S^1)\times D^2
\end{equation*}with $I^+=(a^+,b^+)$ and $b^+<2\pi$ (see figure \ref{figI}). For a compact fiber $\ell_f$, one connects $I_0$ and $I_{2\pi}$ via $0\sim 2\pi$, yielding $V_0$ with only one connected component.

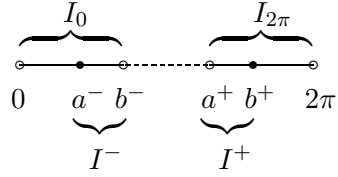
\begin{figure}[h]
\centering
\setlength{\unitlength}{3em}
\begin{picture}(3.7,3)
\put(0.6,1.7){$I_0$}
\put(0.1,1.25){$\overbrace{\phantom{mmmm}}$}
\put(0,0.7){$0$}
\put(0.1,1.2){\circle{0.1}}
\put(0.72,0.6){$\underbrace{\phantom{mm}}$}
\put(0.7,0.7){$a^-$}
\put(0.8,1.2){\circle*{0.1}}
\put(1.2,0.7){$b^-$}
\put(1.3,1.2){\circle{0.1}}
\put(0.9,0){$I^-$}
\put(0.1,1.2){\line(1,0){1.2}}
\multiput(1.325,1.2)(0.1,0){10}{\line(1,0){0.05}}
\put(2.8,1.7){$I_{2\pi}$}
\put(2.3,1.25){$\overbrace{\phantom{mmmm}}$}
\put(2.2,0.6){$\underbrace{\phantom{mm}}$}
\put(2.4,0){$I^+$}
\put(2.3,1.2){\line(1,0){1.2}}
\put(2.2,0.7){$a^+$}
\put(2.3,1.2){\circle{0.1}}
\put(2.7,0.7){$b^+$}
\put(2.8,1.2){\circle*{0.1}}
\put(3.4,0.7){$2\pi$}
\put(3.5,1.2){\circle{0.1}}
\end{picture}
\caption{Intervals along a focus-focus fiber with only one node.}\label{figI}
\end{figure}

Let us represent the trivial regular Lagrangian fibrations as a product of two cotangent bundles
\begin{equation*}
V_0\cong\left(T^*(I_0\sqcup I_{2\pi})\right)\times (I_s\times S^1) \ ,
\end{equation*}
\begin{equation*}
W^-\cong T^*I^-\times (I_s\times S^1) \ ,
\end{equation*}and
\begin{equation*}
W^+\cong T^*I^+\times (I_s\times S^1) \ .
\end{equation*}We do so in order to use lemma \ref{corosolha}, theorem \ref{poincaresolha}, and proposition \ref{kunn}, which give:
\begin{equation*}
H^0(V;\mathcal{J}\big{|}_V)=\{0\} \ ,
\end{equation*}
\begin{equation*}
H^0(W;\mathcal{J}\big{|}_{W})=H^1(W;\mathcal{J}\big{|}_{W})=H^2(W;\mathcal{J}\big{|}_{W})=\{0\} \ ,
\end{equation*}
\begin{equation*}
H^0(V_0;\mathcal{J}\big{|}_{V_0})=H^0(W^-;\mathcal{J}\big{|}_{W^-})=H^0(W^+;\mathcal{J}\big{|}_{W^+})=\{0\} \ ,
\end{equation*}
\begin{equation*}
H^2(V_0;\mathcal{J}\big{|}_{V_0})=H^2(W^-;\mathcal{J}\big{|}_{W^-})=H^2(W^+;\mathcal{J}\big{|}_{W^-})=\{0\} \ ,
\end{equation*}
\begin{equation*}
H^1(V_0;\mathcal{J}\big{|}_{V_0})\cong\left\{\begin{array}{l}
\{0\}\text{, if $\ell_f$ is not Bohr--Sommerfeld} \\
C^\infty(\CR;\CC)\oplus C^\infty(\CR;\CC)\text{, if $\ell_f$ is non-compact}\\
C^\infty(\CR;\CC)\text{, if $\ell_f$ is compact}
\end{array}\right. \ ,
\end{equation*}
\begin{equation*}
H^1(W^-;\mathcal{J}\big{|}_{W^-})\cong\left\{\begin{array}{ll}
C^\infty(\CR;\CC) & \text{, if $\ell_f$ is Bohr--Sommerfeld} \\
\{0\} & \text{, otherwise} \end{array}\right. \ ,
\end{equation*}and
\begin{equation*}
H^1(W^+;\mathcal{J}\big{|}_{W^+})\cong\left\{\begin{array}{ll}
C^\infty(\CR;\CC) & \text{, if $\ell_f$ is Bohr--Sommerfeld} \\
\{0\} & \text{, otherwise} \end{array}\right. \ .
\end{equation*}

Considering the open covering $V_0,W$ of $V$, one has the following exact sequence from the Mayer--Vietoris sequence:
\begin{multline*}
0\rightarrow H^0(V_0\cup W;\mathcal{J}\big{|}_{V_0\cup W}) \\
\rightarrow H^0(V_0;\mathcal{J}\big{|}_{V_0})\oplus H^0(W;\mathcal{J}\big{|}_W)\rightarrow H^0(V_0\cap W;\mathcal{J}\big{|}_{V_0\cap W}) \\
\rightarrow H^1(V_0\cup W;\mathcal{J}\big{|}_{V_0\cup W}) \\
\rightarrow H^1(V_0;\mathcal{J}\big{|}_{V_0})\oplus H^1(W;\mathcal{J}\big{|}_W)\rightarrow H^1(V_0\cap W;\mathcal{J}\big{|}_{V_0\cap W}) \\
\rightarrow H^2(V_0\cup W;\mathcal{J}\big{|}_{V_0\cup W}) \\
\rightarrow H^2(V_0;\mathcal{J}\big{|}_{V_0})\oplus H^2(W;\mathcal{J}\big{|}_W)\rightarrow H^2(V_0\cap W;\mathcal{J}\big{|}_{V_0\cap W}) \\
\rightarrow H^3(V_0\cup W;\mathcal{J}\big{|}_{V_0\cup W})\rightarrow\cdots
\end{multline*}Exploiting the dimension of $V$ (cohomology groups in degree higher than two vanish) and the fact that the cohomology groups in degree zero vanish, as well as in degree two for $V_0$, $W$, and $V_0\cap W$, one has
\begin{multline*}
0\rightarrow H^1(V_0\cup W;\mathcal{J}\big{|}_{V_0\cup W}) \\
\hookrightarrow H^1(V_0;\mathcal{J}\big{|}_{V_0})\oplus H^1(W;\mathcal{J}\big{|}_W)\rightarrow H^1(V_0\cap W;\mathcal{J}\big{|}_{V_0\cap W}) \\ \twoheadrightarrow H^2(V_0\cup W;\mathcal{J}\big{|}_{V_0\cup W})\rightarrow 0 \ .
\end{multline*}

Because $H^1(W;\mathcal{J}\big{|}_{W})=\{0\}$, the middle map
\begin{equation*}
H^1(V_0;\mathcal{J}\big{|}_{V_0})\oplus H^1(W;\mathcal{J}\big{|}_W)\longrightarrow H^1(V_0\cap W;\mathcal{J}\big{|}_{V_0\cap W})
\end{equation*}is injective. This can be seen by identifying the pertinent cohomology groups with $C^\infty(\CR;\CC)$ (see proposition \ref{kunn} and comments below it); thus, the map can be identified with
\begin{equation*}
C^\infty(\CR;\CC)\oplus\{0\}\ni h\oplus 0\mapsto h\oplus h\in C^\infty(\CR;\CC)\oplus C^\infty(\CR;\CC) \ ,
\end{equation*}when the fiber is compact and Bohr--Sommerfeld, and
\begin{equation*}
C^\infty(\CR;\CC)\oplus C^\infty(\CR;\CC)\oplus\{0\}\ni h_1\oplus h_2\oplus 0\mapsto h_1\oplus h_2\in C^\infty(\CR;\CC)\oplus C^\infty(\CR;\CC) \ ,
\end{equation*}when the fiber is Bohr--Sommerfeld but not compact. From the exactness of the sequence and using the first isomorphism theorem, this implies the following:
\begin{equation*}
H^1(V_0\cup W;\mathcal{J}\big{|}_{V_0\cup W})=\{0\}
\end{equation*}and
\begin{equation*}
H^2(V_0\cup W;\mathcal{J}\big{|}_{V_0\cup W})\cong\frac{H^1(V_0\cap W;\mathcal{J}\big{|}_{V_0\cap W})}{H^1(V_0;\mathcal{J}\big{|}_{V_0})} \ .
\end{equation*}

Thus, identifying $V_0\cup W=V$ and $V_0\cap W=W^-\sqcup W^+$, a nodal point on a compact Bohr--Sommerfeld focus-focus fiber provides an infinite dimensional contribution,
\begin{equation*}
H^1(V;\mathcal{J}\big{|}_V)=\{0\}
\end{equation*}and
\begin{equation*}
H^2(V;\mathcal{J}\big{|}_V)\cong\left\{\begin{array}{ll}
C^\infty(\CR;\CC) & \text{, if $\ell_f$ is compact and Bohr--Sommerfeld} \\
\{0\} & \text{, otherwise}
\end{array}\right. \ .
\end{equation*}
\end{proof}

It is possible to adapt the argument of the proof of theorem \ref{ffalive1} to include more nodal points to one focus-focus fiber.

\begin{theorem}\label{ffalive}The geometric quantization of a saturated neighborhood of a focus-focus fiber (with any number of nodes) vanishes if the singular fiber is not Bohr--Sommerfeld, or it is isomorphic to
\begin{equation*}
\bigoplus_{j\in\{1,\dots,n_f\}}C^\infty(\CR;\CC) \ ,
\end{equation*}with $n$ being the number of singular points in the Bohr--Sommerfeld focus-focus fiber, and $n_f=n$ if the focus-focus fiber is compact, or $n_f=n-1$ if the focus-focus fiber is non-compact.
\end{theorem}
\begin{proof}(Sketch)
Instead of one singular point we now have $n$ singularities $p_1,\dots,p_n\in\ell_f$ on the focus-focus fiber. As before we take $W_1,\dots,W_n\subset V$ contractible open neighborhoods of the singular points such that $W_j\cap W_k=\emptyset$ for $j\neq k$, $V_0\subset V$ an open (not connected, even if the fiber is compact) neighborhood satisfying $p_1,\dots,p_n\notin V_0$, as well as, $V=V_0\cup W_1\cup\cdots\cup W_n$, and $V_0\cap W_j=W^-_j\sqcup W^+_j$ for each $W_j$.

Therefore, one can apply the Mayer--Vietoris argument considering the open covering $V_0,V_n=W_1\cup\dots\cup W_n$ of $V$. In place of $W$ one has $V_n$ and since $V_0\cap V_n=\displaystyle\bigsqcup_{j\in\{1,\dots,n\}}W^-_j\sqcup W^+_j$:
\begin{equation*}
H^1(V_0;\mathcal{J}\big{|}_{V_0})\cong\left\{\begin{array}{cl}
\{0\} & \text{, if $\ell_f$ is not Bohr--Sommerfeld} \\
\displaystyle\bigoplus_{j\in\{1,\dots,n+1\}}C^\infty(\CR;\CC) & \text{, if $\ell_f$ is non-compact} \\
\displaystyle\bigoplus_{j\in\{1,\dots,n\}}C^\infty(\CR;\CC) & \text{, if $\ell_f$ is compact}
\end{array}\right. \
\end{equation*}and
\begin{multline*}
H^1(V_0\cap V_n;\mathcal{J}\big{|}_{V_0\cap V_n})\cong\displaystyle\bigoplus_{j\in\{1,\dots,n\}}H^1(W^-_j;\mathcal{J}\big{|}_{W^-_j})\oplus H^1(W^+_j;\mathcal{J}\big{|}_{W^+_j}) \\
\cong\left\{\begin{array}{cl}
\displaystyle\bigoplus_{j\in\{1,\dots,n\}}C^\infty(\CR;\CC)\oplus C^\infty(\CR;\CC) & \text{, if $\ell_f$ is Bohr--Sommerfeld} \\
\{0\} & \text{, otherwise}
\end{array}\right. \ .
\end{multline*}Finally, similarly to the case $n=1$, the map
\begin{equation*}
H^1(V_0;\mathcal{J}\big{|}_{V_0})\oplus H^1(V_n;\mathcal{J}\big{|}_{V_n})\rightarrow H^1(V_0\cap V_n;\mathcal{J}\big{|}_{V_0\cap V_n})
\end{equation*}is injective.
\end{proof}


\section{Semitoric systems and almost toric manifolds}

As for the quantization of Lagrangian fibrations and locally toric manifolds, quantization of neighborhoods of Bohr--Sommerfeld fibers computes the quantization of the whole manifold \cite{Solha} (cf. \cite{Ha} and \cite{Sni}). Consequently, mimicking the Lagrangian bundle and locally toric cases, theorem \ref{ffalive} together with the factorization tools from \cite{Solha} provide the following results.

\begin{theorem}\label{GQalmosttoric}For a $4$-dimensional closed almost toric manifold $M$, with $BS_r$ and $BS_{f}$ the images of the regular and focus-focus Bohr--Sommerfeld fibers on the base:
\begin{equation*}
\mathcal{Q}(M)\cong\left(\bigoplus_{p\in BS_r}\CC\right)\oplus\left(\bigoplus_{p\in BS_{f}}\oplus_{n(p)} C^\infty(\CR ;\CC)\right) \ ,
\end{equation*}with $n(p)$ the number of nodes on the fiber whose image is $p\in BS_{f}$.
\end{theorem}

Closed almost toric manifolds in dimension four were classified in \cite{LeSy}, and in order to obtain their real geometric quantization is enough to identify the image of the Bohr--Sommerfeld fibers at each of the seven possible bases, and then apply theorem \ref{GQalmosttoric}. The total number of regular Bohr--Sommerfeld fibers is determined by the symplectic volume of the almost toric manifold, and the number of focus-focus fibers can be read from table 1 in \cite{LeSy}. Via nodal slides is always possible to modify the real polarization to change the number of focus-focus fibers that are actually Bohr--Sommerfeld fibers; this is exemplified in subsection \ref{glue} for the K3 surface.

Semitoric sytems \cite{pelayo4} are a particular example of almost toric manifolds in dimension four, their bases are subsets of $\CR^2$. However, their total spaces need not to be closed symplectic manifolds; therefore, there is no upper bound on the number of Bohr--Sommerfeld fibers, although their image still form a countable set on the base. Their real geometric quantization is essentially the same as in the case of closed almost toric manifolds, the formula of theorem \ref{GQalmosttoric} holds, but the sets $BS_r$ and $BS_{f}$ can be countably infinite.


\section{Four dimensional examples}

Let us compute the geometric quantization for some specific examples in this last section.


\subsection{K3 surface}\label{glue}

As mentioned in the construction of a pre-quantizable K3 surface (section \ref{KKK}), one can obtain a K3 surface with up to 24 Bohr--Sommerfeld focus-focus fibers. In the particular example constructed in section \ref{KKK}, an application of theorem \ref{GQalmosttoric} yields
\begin{equation*}
\mathcal{Q}(K3)\cong\CC^{14}\oplus\displaystyle\bigoplus_{j\in\{1,\dots,24\}} C^\infty(\CR;\CC) \ .
\end{equation*}

The real geometric quantization of the K3 surface can be, then, drastically different from the K\"{a}hler case, which is always finite dimensional.

On the K3 surface, the dimension of the vector space of holomorphic sections for a given ample holomorphic line bundle $L$ equals $\frac{1}{2}c_1(L)^2+2$ (cf. \cite{Hu}), and the dimension of its K\"{a}hler quantization is exactly this number. Since the first Chern class of a pre-quantum line bundle $L$ is represented by the symplectic form $\omega$, it holds that
\begin{equation*}
c_1(L)^2=\displaystyle\int_{K3}\omega\wedge\omega \ .
\end{equation*}

In the particular example above $K3=(\CC P^2\# 9\overline{\CC P}^2)\#_{T^2}(\CC P^2\# 9\overline{\CC P}^2)$, and the symplectic volume of a symplectic sum is the sum of the symplectic volumes \cite{Gom}. Thus, the symplectic volume can be computed from the symplectic volume of each toric manifold (as nodal trades produce symplectomorphic manifolds \cite{LeSy}), which is simply two times the Euclidean volume of each Delzant polytope \cite{Gu} (up to a $(2\pi)^2$ factor due to different conventions); and the volume of the Delzant polytopes are $24$ in this case. Therefore, for the particular example computed above, the dimension of its K\"{a}hler quantization is
\begin{equation*}
\frac{1}{2}c_1(L)^2+2=\frac{1}{2}(2\cdot 24+2\cdot 24)+2=50 \ .
\end{equation*}

But even when the real geometric quantization is finite dimensional, for this symplectic $K3$ one would have $\mathcal{Q}(K3)\cong\CC^{38}$, which is still different from $\CC^{50}$. This difference is due to how real and K\"{a}hler quantization behave with respect to singular Bohr--Sommerfeld elliptic fibers of a toric manifold. In the real case those fibers (that lie on the the boundary of the Delzant polytope) do not contribute to geometric quantization \cite{Ha}, while they do contribute in the K\"{a}hler case. This means that the real geometric quantization has a simpler behavior under symplectic sum ($19+19=38$) than the K\"{a}hler quantization ($31+31=50+12$).

What is missing is to actually show how to glue pre-quantum line bundles when performing a symplectic sum. We begin by reviewing the symplectic sum construction, and we, then, keep track of this construction when considering pre-quantum line bundles in the picture.
\begin{lemma}[Gompf \cite{Gom}]\label{gomlemma}
Let $(M_j, \omega_j)$, $j=1,2$, be two symplectic manifolds. Assume that there are two codimension $2$ symplectic submanifolds $\Sigma_j \subset M_j$, a symplectomorphism $\Psi: \Sigma_1 \to \Sigma_2$ and a complex isomorphism identifying the symplectic normal bundle $\nu_1$ of $\Sigma_1$ and the dual symplectic normal bundle $\nu_2^*$ of $\Sigma_2$. Then, there is a symplectic structure on the fiber connected sum $M_1 \#_{\Psi} M_2$ of $M_1$ and $M_2$ along $\Sigma_2\simeq \Psi(\Sigma_1)$.
\end{lemma}

Denote by $U_j$ a small tubular neighborhood of $\Sigma_j$ and $U_j^*=U_j \setminus \Sigma_j$. Recall that Gompf's construction provides a symplectomorphism $\Phi:U_1^* \to U_2^*$ that takes the outer boundary of one domain to the inner boundary of the other one and vice versa. This is used as gluing morphism. Let us upgrade the previous Lemma.

\begin{lemma}\label{Franlemma}
Let $(M_j, \omega_j, L_j, \nabla_j)$, $j=1,2$, be two symplectic manifolds equipped with pre-quantum line bundles.
Under the hypotheses of lemma \ref{gomlemma}, assume moreover that $\Psi^* L_2 \simeq L_1$ (as topological complex line bundles), then the fiber connected sum $M_1 \#_{\Psi} M_2$ admits a pre-quantum line bundle $(L, \nabla)$ whose restriction to $M_j\setminus U_j$ coincides with $(L_j, \nabla_j)$
\end{lemma}

\begin{proof}
Since $U_j$ deform retracts to $\Sigma_j$, we have that $\Phi^* L_2 \simeq L_1$ (the topological isomorphism can be extended to $U_j$ and then restricted to $U_j^*$). Denote $(\tilde{L}_1, \tilde{\nabla}_1) = (\Phi^* L_2, \Phi^* \nabla_2)$. We have that over $U_1^*$ the two bundles also satisfy that $curv(\tilde{\nabla}_2) = -\Phi^* i\omega_2 =- i\omega_1 =curv(\nabla_1)$.

Denote by $(\bar{L}_1, \bar{\nabla}_1)$ the dual vector bundle of $(L_1, \nabla_1)$.The bundle $V= \tilde{L}_1 \otimes \bar L_1$, defined over $U_1^*$, is topologically trivial. Moreover, it is equipped with a flat connection $\nabla_V= \tilde{\nabla}_1 - \nabla_1$, therefore we have that $\tilde{\nabla}_1 = \nabla_1 + e^{if}$, for some smooth function $f: U_1^* \to \CR$.

Take a slightly smaller open neighborhood $V_1^* \subset U_1^*$ and the corresponding $\Phi(V_1^*)= V_2^* \subset U_2^*$. Define a function $g:U_1^* \to \CR$ such that $f|_{V_1^*}= g|_{V_1^*}$ and $g$ is compactly supported in $U_1^*$. Define a gauge equivalent connection $\nabla_1' = \nabla_1 +e^{ig}$. By construction, it coincides with $\nabla_1$ away from $U_1^*$ and it coincides with $\Phi^* \nabla_2$ over $V_1^*$. We conclude by gluing $M_1$ and $M_2$ over $V_2^*= \Phi(V_1^*)$.
\end{proof}

\begin{corollary}\label{Francoro}
With the same hypothesis, further assume that there exists a symplectomorphism $\psi:(M,\omega_1,L_1,\nabla_1)\to (M,\omega_2,L_2,\nabla_2)$ with $\psi$ isotopic to the identity ($M=M_1=M_2$), then for any $\Sigma$ we have that ${L_1}|_{\Sigma}$ and ${L_2}|_{\Sigma}$ are isomorphic as topological complex line bundles
\end{corollary}
Do note that $\psi(\Sigma)\neq \Sigma$ in general. If this were not the case, the statement and the gluing would be trivial.
\begin{proof}
Denote by $\psi_t$ the isotopy connecting $\psi_1=id$ with $\psi_2=\psi$. We define $L_t= (\psi_t)_* L_1$ and $\Sigma_t= \psi_t(L_1)$. Therefore, the bundles ${L_2}|_{\Sigma_2}$ and ${L_1}|_{\Sigma_1}$ are isomorphic. Now, since $\Sigma_2$ and $\Sigma_1$ are isotopic submanifolds then the bundles ${L_2}|_{\Sigma_2}$ and ${L_2}|_{\Sigma_1}$ are isomorphic.
\end{proof}


\subsection{Spin-spin system}

Another example (which may be considered as a toy model to the spin-spin system of \cite{SaZh}) is to consider the product of two spheres. As a toric manifold it is represented by its Delzant polytope (which is a square), and one can perform a nodal trade on one of its vertices: see figure \ref{figS2}.

\begin{figure}[h]
\centering
\setlength{\unitlength}{2em}
\begin{picture}(6,2)
\put(0,0){\circle*{0.2}}
\put(2,0){\circle*{0.2}}
\put(2,2){\circle*{0.2}}
\put(0,2){\circle*{0.2}}
\multiput(0,0)(1,0){3}{\circle*{0.1}}
\multiput(0,1)(1,0){3}{\circle*{0.1}}
\multiput(0,2)(1,0){3}{\circle*{0.1}}
\put(0,0){\line(1,0){2}}
\put(2,0){\line(0,1){2}}
\put(2,2){\line(-1,0){2}}
\put(2,2){\vector(-1,-1){0.9}} 
\put(0,2){\line(0,-1){2}}
\put(3,1){\vector(1,0){2}}
\end{picture}
\begin{picture}(2,2)
\put(0,0){\circle*{0.2}}
\put(2,0){\circle*{0.2}}
\put(0,2){\circle*{0.2}}
\put(1,1){\circle*{0.2}}
\multiput(0,0)(1,0){3}{\circle*{0.1}}
\multiput(0,1)(1,0){3}{\circle*{0.1}}
\multiput(0,2)(1,0){2}{\circle*{0.1}}
\put(0,0){\line(1,0){2}}
\put(2,0){\line(0,1){1.5}}
\put(1.5,2){\line(-1,0){1.5}}
\put(1,1){\circle{0.3}} \qbezier(2,1.5)(2,2)(1.5,2) 
\put(0,2){\line(0,-1){2}}
\end{picture}
\caption{Nodal trade on $S^2\times S^2$.}\label{figS2}
\end{figure}
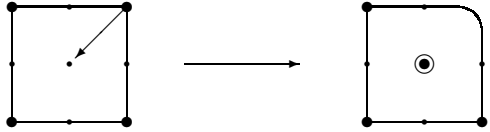

Considering the two spheres as unit spheres in $\CR^3$ (with the induced area form as the symplectic structures) parametrized by coordinate functions $(x_j,y_j,z_j)$, with $j=1,2$ and $x_j^2+y_j^2+z_j^2=1$, the toric system of figure \ref{figS2} is described by the first integrals $f_1=z_1$ and $f_2=z_2$, while the spin-spin system of \cite{SaZh} is descibed by the first integrals $\displaystyle f_1=\frac{z_1}{2}+\displaystyle\frac{x_1x_2+y_1y_2+z_1z_2}{2}$ and $f_2=z_1+z_2$, whose qualitative traits are represented by the semitoric system obtained from the above toric system via a nodal trade.

The resulting construction yields,
\begin{equation*}
\mathcal{Q}(S^2\times S^2)\cong C^\infty(\CR;\CC)
\end{equation*}for the semitoric system.


\subsection{Spherical pendulum and the spin-oscillator system}

Particular choices of the parameters on the system of a spherical pendulum (mass and length of the pendulum, as well as the gravitational acceleration) imply that the focus-focus singular fiber is Bohr--Sommerfeld. As a result, the eigenspace attached to the eigenvalue given by the energy of the classical unstable equilibrium state of a focus-focus singular point might be infinite dimensional.

The coupled classical spin and harmonic oscillator of \cite{pelayo2} (a classical version of the Jaynes--Cummings model \cite{JaCu}) also behaves in a similar fashion. While the spherical pendulum is described by an integrable system on the cotangent bundle of the sphere, the spin-oscillator system is described by a semitoric system on the trivial bundle $\CR^2\times S^2$ over the sphere, whose first integrals are $f_1=z+\frac{1}{2}(u^2+v^2)$ and $f_2=\frac{1}{2}(xu+yv)$, where $(u,v)$ are the coordinate functions on the fibers $\CR^2$ and $(x,y,z)$ satisfying $x^2+y^2+z^2=1$ are the coordinates on the unit sphere; the symplectic structure is the product one, with the one in the fibers being $\ud u\wedge\ud v$ and the one in the sphere being the induced area form from the Euclidean $\CR^ 3$.

This sort of degeneracy at energy values associated with classical unstable equilibrium states does not represent the physics of the quantum systems (cf. \cite{Condon,CuDu}), and it would be interesting to refine the definition of geometric quantization with singular real polarizations to mod out these infinite dimensional contributions to get a finite dimensional Hilbert space instead.

It is worth mentioning, though, that this infinite dimensional degeneracy on these examples is not generic. If the singular Lagrangian fibration is perturbed so the focus-focus fibers are no longer Bohr--Sommerfeld (e.g. by slightly changing the parameters of a spherical pendulum, by performing a nodal slide on almost toric manifolds, or by perturbing the connection in the pre-quantum line bundle), their real geometric quantisation ends up being finite dimensional.


\section{Higher dimensions}

We may iterate these techniques to obtain higher dimensional results. We would like to stress out that there is no complete symplectic topological classification of integrable systems in a neighbourhood of semitoric fibers in dimension greater than 4. However, by applying the Mayer--Vietoris argument used in the proof of theorem \ref{ffalive1} together with the K\"{u}nneth formulae from \cite{MiPr}, we conjecture:

\begin{conjecture}
Let $M$ is a $2m$-dimensional closed almost toric manifold with singularities of any corank and of Williamson type $(k_e, 0, k_f)$ with $k_f< 2$:
\begin{equation*}
\mathcal{Q}(M)\cong\left(\bigoplus_{p\in BS_r}\CC\right)\oplus\left(\bigoplus_{p\in BS_{f-r}}\oplus_{n(p)} C^\infty(\CR ;\CC)\right) \ ,
\end{equation*}where $BS_r$ and $BS_{f-r}$ are the images of the regular and focus-focus-regular (Williamson type $(0,0,1)$) Bohr--Sommerfeld fibers on the base, and $n(p)$ the number of nodes on the fiber whose image is $p\in BS_{f-r}$.
\end{conjecture}

The conjecture above  does not take into account singular Lagrangian fibrations with hyperbolic components but those can be considered using the results in  \cite{HaMi}:

\begin{conjecture}
For non-degenerate integrable systems with singularities of any corank and of Williamson type $(k_e, k_h, h_f)$ with $k_h< 2$, $k_f<2$, and $k_h+k_f\leq 1$ on a closed symplectic manifold:
\begin{equation*}
\mathcal{Q}(M)\cong\left(\bigoplus_{p\in BS_r}\CC\right)\oplus\left(\bigoplus_{p\in BS_{f-r}}\oplus_{n(p)} C^\infty(\CR ;\CC)\right)\oplus\left(\bigoplus_{p\in BS_{h-r}}\CC^{\CN}\oplus\CC^{\CN}\right) \ .
\end{equation*}
\end{conjecture}

Due to the infinite dimensional contributions coming from hyperbolic and focus-focus singularities, the analogous results containing mixed product of these components are not considered here, e.g. Williamson types $(0,2,0)$, $(0,1,1)$, $(0,0,2)$, or higher dimensional analogues of any codimension. In these examples a completion might be needed when considering the K\"{u}nneth formulae (see \cite{MiPr,kaup,grothendieck}).



\end{document}